\documentclass[11pt]{article}

\usepackage[T1]{fontenc}
\usepackage{amsmath,amsthm,amssymb,eqnarray}
\usepackage[latin1]{inputenc}
\usepackage{fourier} 
\usepackage[english]{babel} 
\usepackage{tikz}
\usetikzlibrary{decorations.pathreplacing}
\usetikzlibrary{decorations.markings}
\usepackage{mathtools} 
\usepackage{hyperref}
\usepackage[margin=1in]{geometry} 
 \usepackage{enumerate}
\usepackage{graphicx}
\hypersetup{
    colorlinks=true,
    linkcolor=blue,
    filecolor=magenta,      
    urlcolor=blue,
}
\usepackage{fancyvrb}
\usepackage{chngcntr}
\counterwithin{figure}{section}
\counterwithin{table}{section}
\counterwithin{equation}{section}
\usepackage{sectsty}
\usepackage{relsize}

\usepackage{color}
\usepackage[framed,numbered,autolinebreaks,useliterate]{mcode}

\theoremstyle{plain}

\newtheorem{theorem}{Theorem}[section]
\newtheorem{lp}[theorem]{Linear Program}
\newtheorem{ip}[theorem]{Integer Program}
\newtheorem{milp}[theorem]{Mixed Integer Linear Program}
\newtheorem{lemma}[theorem]{Lemma}

\newtheorem{conjecture}[theorem]{Conjecture}
\newtheorem{remark}[theorem]{Remark}
\newtheorem{code}{Code}[section]

\theoremstyle{remark}

\newcommand{\Z}{\mathbb{Z}}

\newcommand{\K}{\mathcal{K}}
\renewcommand{\S}{\mathcal{S}}

\newcommand{\exd}{\gamma^*_e}

\DeclareMathOperator{\exc}{exc}
\DeclareMathOperator{\m}{m}
\DeclareMathOperator{\dist}{dist}

\title{A Linear programming method for exponential domination} 
\author{Michael Dairyko$^1$ \and Michael Young$^1$}

\begin{document}
\maketitle
\footnotetext[1]{Department of Mathematics, Iowa State University, Ames, IA 50011, USA. (mdairyko, myoung) @iastate.edu}

\begin{abstract}
For a graph $G,$ the set $D \subseteq V(G)$ is a porous exponential dominating set if  $1 \le \sum_{d \in D} \left( 2 \right)^{1-\dist(d,v)}$ for every $v \in V(G),$ where $\dist(d,v)$ denotes the length of the shortest $dv$ path. The porous exponential dominating number of $G,$ denoted $\exd(G),$ is the minimum cardinality of a porous exponential dominating set. For any graph $G,$ a technique is derived to determine a lower bound for $\exd(G).$ Specifically for a grid graph $H,$ linear programing is used to sharpen bound found through the lower bound technique. Lower and upper bounds are determined for the porous exponential domination number of the King Grid $\K_n,$ the Slant Grid $\S_n,$ and the $n$-dimensional hypercube $Q_n.$  \\



\end{abstract}

\begin{section}{Introduction}
Domination in graphs is a tool used to model situations in which a vertex exerts influence on its neighboring vertices. For a graph $G,$ a set $D \subseteq V(G)$ is a \emph{dominating set} if every vertex contained in $V(G) \setminus D$ is adjacent to at least one vertex of $D.$ The \emph{domination number}, denoted $\gamma(G),$ is the cardinality of a minimum domination set.

Exponential domination was first introduced in \cite{dankel} and is a variant of domination that models situations in which the influence of an object exerts decreases exponentially as the distance increases. In particular exponential domination models the dissemination of information in social networks where the information's influence decays exponentially with each share  \cite{dankel}. Therefore, exponential domination analyzes objects with a global influence. Other variants of domination investigate objects with local influence. There are two parameters within exponential domination; porous and non-porous. This paper focuses on porous exponential domination. A \emph{porous exponential dominating set} is a set $D \subseteq V(G)$ such that $w^*(D,v) \ge 1$ for every $v \in V(G),$ where the weight function $w^*$ is given by $w^*(u,v) =2^{ 1-\dist(u,v)}$ and $\dist(u,v)$ represents the length of the shortest $uv$ path.  The \emph{porous exponential domination number} of $G$, denoted by $\exd(G),$ is the cardinality of a minimum porous exponential dominating set. For the sake of simplicity, we will refer to porous exponential domination as exponential domination. See Section \ref{Prelim} for technical definitions.

Section \ref{GenTech} develops a technique to determine the lower bound of the exponential domination number of any graph. Furthermore, with respect to grid graphs, a method using linear programing sharpens the lower bound. Section \ref{MainProofs} applies the lower bound technique described in Section \ref{GenTech}, to find lower bounds for the exponential domination number of the King grid $\K_n,$ the Slant grid $\S_n,$ and the $n$-dimensional hypercube $Q_n$. Upper bound constructions are then found for $\exd(\K_n),$ $\exd(\S_n)$ and $\exd(Q_n).$ 

\begin{subsection}{Preliminaries}\label{Prelim}
All graphs are simple and undirected. A graph $G = (V(G), E(G))$ is an ordered pair that is formed by a set of \emph{vertices} $V(G)$ and a set of \emph{edges} $E(G),$ where an edge is the two element subset of vertices. For the two sets $A$ and $B,$ the Cartesian product of $A$ and $B$ is defined to be $A\times B = \{ (a,b) : a\in A \text{ and } b \in B \}.$ Consider the graph $G$ and the set $D \subset V(G).$  Let $w:V(G)\times V(G)\rightarrow \mathbb{R}$ be a \emph{weight function}. For $u,v \in V(G),$ we say that $u$ assigns weight $w(u,v)$ to $v.$ Denote the weight assigned by $D$ to $v$ as $w(D,v)  := \sum_{d \in D} w(d,v),$ and similarly, the weight assigned by $d\in D$ to $H \subseteq V(G)$ as $w(d,H)  := \sum_{h \in H} w(d,h).$ Let $\m(G) =  \max_{d\in D} w( d, V(G) ).$ The pair $(D,w)$ dominates $G$ if $w(D,v) \ge 1$ for all $v\in V(G).$ The \emph{excess weight} that the vertex $v$ receives from $D$ is defined as $\exc(D,v) = w(D, v) - 1.$ We denote $\exc(D) = \sum_{v\in V(G)} \exc(D,v)$ to be the total excess weight that $D$ sends out. Let $S_k(v) = \{ u \in V(G) : \dist(u,v) = k \}$ denote the sphere of radius $k.$  

Linear programing is an optimization technique that takes a set of linear inequalities, or constraints, and finds the best solution of a linear objective function. An integer program is a linear program, with the restriction the variables can only be assigned integer values. Observe that $\exd(G)$ is equivalent to finding the optimal value of the following integer program introduced by Henning et al:

\begin{ip}\cite{Henning}\label{IP}
\begin{eqnarray*}
\min \sum\limits_{u \in V(G) } x(u) &&  \\
\text{s.t.}  \sum\limits_{u \in V(G) } \left( \frac{1}{2} \right) ^{\dist(u,v) -1 }  x(u) &\ge& 1\quad  \forall v \in V(G) \\
 x(u) &\in& \{0,1 \} \quad  \forall  u \in V(G). 
\end{eqnarray*}
\end{ip}
Notice that it is only feasible to run the program for graphs of small size, as the computation time for this integer program greatly increases as the size of the graph increases. To be able to run the program on graphs with larger sizes, the constraints in Integer Program \ref{IP} can be relaxed as shown in the following linear program. 
 
\begin{lp}\cite{Henning}\label{LP}
\begin{eqnarray*}
\min \sum\limits_{u \in V(G) } x(u) &&  \\
\text{s.t.}  \sum\limits_{u \in V(G) } \left( \frac{1}{2} \right) ^{\dist(u,v) -1 }  x(u) &\ge& 1 \quad \forall v \in V(G) \\
 x(u) &\ge& 0 \quad \forall u \in V(G). 
\end{eqnarray*}

\end{lp}

The Cartesian product of two graphs $G$ and $H,$ denoted $G\square H,$ is a graph such that $V(G\square H) = V(G) \times V(H)$ and two vertices $(g,h) \sim (g',h')$ in $G\square H$ if and only if either $g = g'$ and $h \sim h'$ in $H,$ or $h=h'$ and $g\sim g'$ in $G.$ Let $G_{m,n} = P_m \square P_n$ be the \emph{standard grid.} A \emph{grid graph} is the standard grid with possibly additional edges added in a regular pattern. Notice that linear programming is a natural technique to apply to grid graphs. Observe that asymptotically, $G_{m,n}$ is equivalent to the torus $C_m \square C_n,$ which yields the same lower bound for the corresponding exponential domination number. 
\begin{figure}[h!tp]
\[ \begin{array}{ccc}
\begin{tikzpicture}[scale=0.7]
\draw[step=1cm] grid (4,4);
\foreach \x/\y/\l in {0/0/a, 0/1/aa, 0/2/aaa, 0/3/aaaa, 0/4/z }
\node[circle, fill = white, draw] (\l) at (\x,\y) {};
\foreach \x/\y/\l in {1/0/b, 1/1/bb, 1/2/bbb, 1/3/bbbb, 1/4/bbbbb }
\node[circle, fill = white, draw] (\l) at (\x,\y) {};
\foreach \x/\y/\l in {2/0/c, 2/1/cc, 2/2/ccc, 2/3/cccc, 2/4/ccccc }
\node[circle, fill = white, draw] (\l) at (\x,\y) {};
\foreach \x/\y/\l in {3/0/d, 3/1/dd, 3/2/ddd, 3/3/dddd, 3/4/ddddd }
\node[circle, fill = white, draw] (\l) at (\x,\y) {};
\foreach \x/\y/\l in {4/0/e, 4/1/ee, 4/2/eee, 4/3/eeee, 4/4/eeeee }
\node[circle, fill = white, draw] (\l) at (\x,\y) {};
\foreach \x/\y in {a/bb, bb/ccc, ccc/dddd, dddd/eeeee }
\draw (\x) -- (\y);
\foreach \x/\y in {aa/bbb, bbb/cccc, cccc/ddddd }
\draw (\x) -- (\y);
\foreach \x/\y in {aaa/bbbb, bbbb/ccccc }
\draw (\x) -- (\y);
\foreach \x/\y in {aaaa/bbbbb }
\draw (\x) -- (\y);
\draw (b) -- (cc); \draw (cc) -- (ddd);\draw (ddd) -- (eeee); \draw (c) -- (dd); \draw (dd) -- (eee); \draw (d) -- (ee);
\draw (bbbb) -- (ccc);\draw (ccc) -- (dd); \draw (dd) -- (e);
\draw (aa) -- (b); \draw (aaa) -- (bb);\draw (aaaa) -- (bbb);
\draw (bb) -- (c); \draw (bbb) -- (cc);
\draw (cc) -- (d); \draw (cccc) -- (ddd);
\draw(ddd) -- (ee); \draw(dddd) -- (eee);
\draw(0,4) -- (1,3); \draw(1,4) -- (2,3); \draw(2,4) -- (3,3); \draw(3,4) -- (4,3);
\node[circle, fill = white, draw] (aaaaa) at (0,4) {};
\node[circle, fill = white, draw] (bbb) at (1,3) {};
\node[circle, fill = white, draw] (cccc) at (2,4) {};
\node[circle, fill = white, draw] (ddd) at (3,3) {};
\node[circle, fill = white, draw] (dddd) at (3,4) {};
\node[circle, fill = white, draw] (eee) at (4,3) {};
\node[circle, fill = white, draw] (bbbb) at (1,4) {};
\node[circle, fill = white, draw] (ccc) at (2,3) {};
\end{tikzpicture}
& 
\begin{tikzpicture}[scale = 1.5]

\node[circle, fill = white, draw] (a) at ( 1,0) {};
\node[circle, fill = white, draw] (b) at ( .71, - .71) {};
\node[circle, fill = white, draw] (c) at ( 0, - 1) {};
\node[circle, fill = white, draw] (d) at ( -.71, - .71) {};
\node[circle, fill = white, draw] (e) at ( -1, 0) {};
\node[circle, fill = white, draw] (f) at ( -.71,  .71) {};
\node[circle, fill = white, draw] (g) at ( 0, 1) {};
\node[circle, fill = white, draw] (h) at ( .71,  .71) {};

\node[circle, fill = white, draw] (a1) at ( 0.4,0) {};
\node[circle, fill = white, draw] (b1) at ( .3, - .3) {};
\node[circle, fill = white, draw] (c1) at ( 0, - 0.4) {};
\node[circle, fill = white, draw] (d1) at ( -.3, - .3) {};
\node[circle, fill = white, draw] (e1) at ( -0.4, 0) {};
\node[circle, fill = white, draw] (f1) at ( -.3,  .3) {};
\node[circle, fill = white, draw] (g1) at ( 0, 0.4) {};
\node[circle, fill = white, draw] (h1) at ( .3,  .3) {};

\draw[ thick,color = blue] (c) -- (b1) ;
\draw[ thick,color = blue] (c) -- (b) ;
\draw[ thick,color = blue] (c) -- (d) ;
\draw[ thick,color = blue] (d) -- (e1) ;
\draw[ thick,color = blue] (e1) -- (b1) ;
\draw[ thick,color = blue] (e1) -- (h1) ;
\draw[ thick,color = blue] (c1) -- (h1) ;
\draw[ thick,color = blue] (h1) -- (a) ;
\draw[ thick,color = blue] (a) -- (b) ;
\draw[ thick,color = blue] (a) -- (b1) ;
\draw[ thick,color = blue] (c1) -- (d) ;
\draw[ thick,color = blue] (c1) -- (b) ;

\draw[ thick,color = red] (e) -- (f) ;
\draw[ thick,color = red] (e) -- (d1) ;
\draw[ thick,color = red] (e) -- (f1) ;
\draw[ thick,color = red] (f) -- (g) ;
\draw[ thick,color = red] (g1) -- (f) ;
\draw[ thick,color = red] (g) -- (f1) ;
\draw[ thick,color = red] (g) -- (h) ;
\draw[ thick,color = red] (h) -- (a1) ;
\draw[ thick,color = red] (g1) -- (d1) ;
\draw[ thick,color = red] (f1) -- (a1) ;
\draw[ thick,color = red] (d1) -- (a1) ;
\draw[ thick,color = red] (g1) -- (h) ;

\draw (a) -- (h);
\draw (g) -- (h1);
\draw (f) -- (e1);
\draw (e) -- (d);
\draw (d1) -- (c);
\draw (a1) -- (b);
\draw (b1) -- (g1);
\draw (c1) -- (f1);
\end{tikzpicture}
&
\begin{tikzpicture}[scale=0.7]
\draw[step=1cm] grid (4,4);
\foreach \x/\y/\l in {0/0/a, 0/1/aa, 0/2/aaa, 0/3/aaaa, 0/4/z }
\node[circle, fill = white, draw] (\l) at (\x,\y) {};
\foreach \x/\y/\l in {1/0/b, 1/1/bb, 1/2/bbb, 1/3/bbbb, 1/4/bbbbb }
\node[circle, fill = white, draw] (\l) at (\x,\y) {};
\foreach \x/\y/\l in {2/0/c, 2/1/cc, 2/2/ccc, 2/3/cccc, 2/4/ccccc }
\node[circle, fill = white, draw] (\l) at (\x,\y) {};
\foreach \x/\y/\l in {3/0/d, 3/1/dd, 3/2/ddd, 3/3/dddd, 3/4/ddddd }
\node[circle, fill = white, draw] (\l) at (\x,\y) {};
\foreach \x/\y/\l in {4/0/e, 4/1/ee, 4/2/eee, 4/3/eeee, 4/4/eeeee }
\node[circle, fill = white, draw] (\l) at (\x,\y) {};
\foreach \x/\y in {a/bb, bb/ccc, ccc/dddd, dddd/eeeee }
\draw (\x) -- (\y);
\foreach \x/\y in {aa/bbb, bbb/cccc, cccc/ddddd }
\draw (\x) -- (\y);
\foreach \x/\y in {aaa/bbbb, bbbb/ccccc }
\draw (\x) -- (\y);
\foreach \x/\y in {aaaa/bbbbb }
\draw (\x) -- (\y);
\draw (b) -- (cc); \draw (cc) -- (ddd);\draw (ddd) -- (eeee); \draw (c) -- (dd); \draw (dd) -- (eee); \draw (d) -- (ee);
\node[circle, fill = white, draw] (aaaaa) at (0,4) {};
\node[circle, fill = white, draw] (bbb) at (1,3) {};
\node[circle, fill = white, draw] (cccc) at (2,4) {};
\node[circle, fill = white, draw] (ddd) at (3,3) {};
\node[circle, fill = white, draw] (dddd) at (3,4) {};
\node[circle, fill = white, draw] (eee) at (4,3) {};
\node[circle, fill = white, draw] (bbbb) at (1,4) {};
\node[circle, fill = white, draw] (ccc) at (2,3) {};
\end{tikzpicture}

\end{array} \]
 \caption{An illustration of $\K_5,$ $Q_4,$ and  $\S_{5}$}
\label{K5S5}
\end{figure} 

The \emph{strong product} of two graphs $G$ and $H$ is the graph $G \boxtimes H $ for which $V(G\boxtimes H) = V(G) \times V(H)$ and two distinct vertices are adjacent whenever in both coordinate places the vertices are adjacent or equal in the corresponding graph. The \emph{King grid} is defined as $\K_n = P_n \boxtimes P_n.$ Let $[n] = \{ 1,2, \ldots, n\}.$ Consider the paths $P_n$ and $P_m$ with vertex sets $[n]$ and $[m],$ respectively. Then the \emph{Slant} grid is defined to be $\S_n = P_n \square P_m$ with the additional edges $\{i,j\} \sim \{i+1, j+1\},$ for $ i \in [n-1]$ and $j \in [m-1].$ Notice that $\K_n$ and $\S_n$ are both instances of grid graphs. The \emph{$n$-dimensional hypercube} graph, denoted $Q_n,$ is constructed by creating a vertex for each $n$-digit binary number. Edges are formed if two vertices differ by one digit in their binary representation.  See Figure \ref{K5S5} for an illustration of $\K_5,$ $Q_4,$ and  $\S_5.$

 \end{subsection}

\subsection{Motivation}
For $m\le n$ consider $C_m \square C_n,$ the \emph{torus} graph. Exponential domination of $C_m \square C_n$ was first studied in \cite{anderson}. Figure \ref{torus13} is a visual representation of $C_{13} \square C_{13},$ where $`X'$ denotes a member of $D,$ an exponential domination set.  Observe that there is one member of $D$ in every row and column, therefore giving an upper bound construction for $\gamma_e(C_m \square C_n)$ when $m$ and $n$ are multiples of $13.$ The following theorem extends this idea to large graphs.
 \begin{figure}[htp!]
\centering 
\begin{tikzpicture}[scale=0.3]
\draw[step=1cm,gray,very thin] grid (13,13);
\foreach \x/\y/\l in {0.5/12.5/a, 1.5/7.5/b, 2.5/2.5/c, 3.5/10.5/d, 4.5/5.5/e, 5.5/0.5/f, 6.5/8.5/g, 7.5/3.5/h, 8.5/11.5/i, 9.5/6.5/j, 10.5/1.5/k, 11.5/9.5/l, 12.5/4.5/m}
\node (\l) at (\x,\y)  {\tiny X};
\end{tikzpicture}
\caption{$13\times 13$ Exponential Dominating Set Tile for $C_\infty \square C_\infty$ }
\label{torus13}
\end{figure}
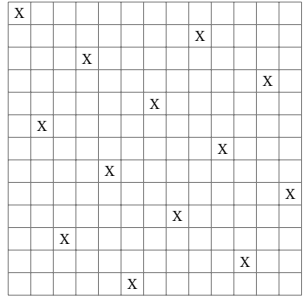

\begin{theorem}\cite{anderson}\label{UBtor}
$\lim_{n\to \infty} \frac{\exd(C_m \square C_n)}{mn} \le \frac{1}{13}.$
\end{theorem}

Notice that Theorem \ref{UBtor} directly implies that for $m,n \ge 13,$ $\exd(C_m\square C_n) \le \left\lceil\frac{mn}{13}\right\rceil  + o(1).$ Through a naive counting argument, it was shown that for $m,n \ge 3,$ $\left\lceil\frac{mn}{15.875}\right\rceil \le \exd(C_m\square C_n)$ \cite{anderson}. These results lead to the following conjecture.
\begin{conjecture}\label{conj13}
For all $m$ and $n,$ $\left\lceil\frac{mn}{13}\right\rceil \le \exd(C_m\square C_n).$
\end{conjecture}

The lower bound for $\exd(C_m \square C_n)$ was improved in \cite{EGR} by taking the counting argument from \cite{anderson} and applying it to linear programming. 

\begin{theorem}\label{bestlowerTorus}\cite{EGR}
For all $m,n \ge 11$, $\left\lceil\frac{mn}{13.761891939197298}\right\rceil \le \exd(C_m\square C_n).$
\end{theorem}

This paper was motivated by the work on determining $\exd(C_m \square C_n)$ from \cite{anderson} and \cite{EGR}. The case specific lower bound technique from \cite{EGR} is generalized to all graphs and the linear programming method detailed in \cite{EGR} is generalized to all grid graphs.

\end{section}

\begin{section}{A Lower Bound Technique  }\label{GenTech}
In this section, a technique for determining the lower bound of the exponential domination number of any graph is derived. Through the use of linear programing, this technique is improved specifically for grid graphs. Note that the bound in Lemma \ref{kbound} is sharp if $w^*(v, V(G) ) =\m(G)$ for every $v\in V(G).$ 

\begin{lemma}\label{kbound}
Let $D$ be an exponential dominating set for the graph $G.$ If $k|D| \le \exc(D),$ then \[ \left\lceil \dfrac{|V(G)|}{\m(G) - k} \right\rceil  \le |D|.\]
\end{lemma}
\begin{proof}
Observe that
\[  |V(G)| \le \sum_{v\in V(G)} w(D,v) = \sum_{d\in D}\sum_{v \in V(G)}  w(d, v ) \le |D|\m(G) - \exc(G)  \le |D| \left(\m(G) -\frac{\exc(D)}{|D|} \right) \le  |D| \left(\m(G) - k \right)\] 
\end{proof}

\begin{remark}\label{LPsetup}{ \rm
In Lemma \ref{kbound}, the value $k$ is needed to compute the lower bound. For grid graphs, linear programming can be used to determine such a value of $k.$ Mixed Integer Linear Program \ref{MILP} is derived through the use of Linear Program \ref{LP} with two additional constraints. See Section \ref{MILPsetup} for the construction details. Let $x_{\min}$ be the optimal solution found from Mixed Integer Linear Program \ref{MILP}. As $w^*(D,v) \ge 1$ for all $v\in V(G),$ it follows that $|I| < x_{\min}.$ Therefore $k =  x_{\min}  - |I|.$ }
\end{remark}
\begin{milp} \label{MILP}
\begin{equation*}
\begin{array}{rrcll}
\min &\displaystyle\sum_{i \in I} [A{\bf x}]_i &&\\
{\rm s.t.} & A{\bf x} &\ge& {\bf 1}\\
& A{ \bf x} &\le& b\\
&{\bf x}& \ge & {\bf 0} &\\
& x_i & \le &2, \; i \in I \\
& x_1 & = & 2.
\end{array}
\end{equation*}
\end{milp} 

\begin{remark}{ \rm
Observe that Remark \ref{LPsetup} localizes the global nature of exponential domination. Recall that exponential domination has a growth factor of $\frac{1}{2}.$ Therefore this method can be applied to the variant of exponential domination with the growth factor of $\frac{1}{p}$ for $p\ge 3.$ Furthermore, the method can be applied to other variants of domination to obtain a lower bound for the corresponding domination number. However, it is unclear whether the lower bound derived will be significant.  }
\end{remark}

\begin{subsection}{Mixed Integer Linear Program Setup}\label{MILPsetup}
The setup for Mixed Integer Linear Program \ref{MILP} is now discussed. Consider the $m\times n$ grid graph $G$ and let $D$ be a corresponding exponential dominating set. For a fixed $d_0 \in D$ and given an odd positive integer $r \le \min \{m,n\},$ define $H$ to be the $r\times r$ subgrid of $G$ centered at $d_0.$ Label the set of vertices $V(H)$ as $\{ v_1, v_2, \ldots, v_{r^2} \}$ and let the indices of the interior vertices of $H$ be defined as \[ I = \left\{  i : v_i \in V(H) \text{ and  } \dist(d_0 , v_i) < \left\lfloor \frac{r}{2} \right\rfloor\right\}.\]  Then for $1\le k \le r^2,$ define $S_k = v_k \cup \{ u \in V(G\setminus H) : \dist(u,v_k) \le \dist(u, h ) \;\forall h\in V(H) \} $ and $x_k = w^*(S_k \cap D, v_k).$ Notice that $S_i = v_i$ for every $i \in I.$ Therefore for $1 \le k,j \le r^2,$ it follows that $w^*(S_k \cap D,v_j) \le x_k \left( \frac{1}{2} \right)^{\dist(v_k,v_j)} .$  Thus, by the construction of $S_k,$ \[ w^*(D, v_j) \le \sum_{k=1}^{r^2} w^*(S_k \cap D, v_j) \le \sum_{k = 1}^{r^2} x_k \left( \frac{1}{2} \right)^{\dist(v_k,v_j)} .\] Let $A$ be the $r^2 \times r^2$ matrix such that $[A]_{kj} = \left(\frac{1}{2} \right) ^ {\dist(v_k,v_j)}.$ Furthermore, let $\vec{x} = [x_1, x_2, \ldots, x_{r^2}]^\intercal,$ where $x_1$ corresponds to $d_0,$ and $\vec{w} = [w^*(D,v_1), w^*(D,v_2), \ldots, w^*(D,v_{r^2})]^\intercal.$ Then observe that ${\vec w} \le A \vec{x}.$ The aim is to minimize $w^*(d_0, v_i)$ for all $i \in I$, while still satisfying that $w^*(D, v_i) \ge 1.$ Therefore the objective function is to minimize  $\sum_{i \in I} [A{\bf x}]_i,$  where ${\bf x}$ is a vector of $r^2$ nonnegative variables.

Let ${\bf 0}$ and ${\bf 1}$ denote the $0s$ and $1s$ vectors of length $r.$ Then the two constraints of Linear Program \ref{LP} with respect to the grid graph construction are that $ A {\bf x} \ge {\bf 1}$ and ${\bf x} \ge {\bf 0}.$ The remaining two constraints of Mixed Integer Linear Program are now discussed. By construction, any member of $D$ assigns itself weight $2,$ and the remaining vertices do not have any initial weight. This gives the first integer constraint that $x_1 = 2$ and $x_i \le 2,$ for  $i \in I.$ Observe that it is necessary to determine an upper bound for $w^*(D,v_i)$ for each $v_i \in V(H)$ so that $w^*(d_0, v_i)$ can be decreased by the appropriate amount. To ensure this, we want 
\[ 0 \le  w^*(d_0, v_i) - \exc(D, v_i)  =  w^*(d_0, v_i) - (w^*(D,v_i) - 1).\] 
This implies that $w^*(D,v_i) \le 1 + w^*(d_0, v_i) .$ Let $b$ be the real valued vector such that $b_i =  1+\left(\frac{1}{2} \right) ^ {\dist(d_0,v_i) -1}$ for $1\le i \le r^2.$ Therefore, the second constraint is $A{\bf x} \le b$.

\end{subsection}

\end{section}

\begin{section}{Main Results}\label{MainProofs}
 In this section the lower bound technique discussed in Section \ref{GenTech} is applied and upper bound constructions are found to bound the exponential domination number of the the King grid $\K_n,$ Slant grid $\S_n,$ and $n$-dimensional hypercube $Q_n.$


\begin{subsection}{The King Grid $\K_n$}\label{King}

For small values of $n,$ the exact value of $\exd(\K_n)$ can be determined using Integer Program \ref{IP}. Figure \ref{minDomsetsKing} visualizes the location of the corresponding exponential dominating vertices for $\exd(\K_n),$  denoted by  `X'. See Code \ref{KingExact} for the corresponding SAGE code.

\begin{figure}[!h]
\[ \begin{array}{ccccccc}
\begin{tikzpicture}[scale=0.3]
\draw[step=1cm,gray,very thin] grid (2,2);
\foreach \x/\y/\l in {0.5/1.5/a}
\node (\l) at (\x,\y)  { {\tiny X}};
\node (a) at (1,-0.75) {\small $n=2$};
\end{tikzpicture}
&
\begin{tikzpicture}[scale=0.3]
\draw[step=1cm,gray,very thin] grid (3,3);
\foreach \x/\y/\l in {1.5/1.5/a}
\node (\l) at (\x,\y)  {{\tiny X}};
\node (a) at (1.5,-0.75) {\small $n=3$};
\end{tikzpicture}
&
\begin{tikzpicture}[scale=0.3]
\draw[step=1cm,gray,very thin] grid (4,4);
\foreach \x/\y/\l in {1.5/2.5/a, 2.5/2.5/b}
\node (\l) at (\x,\y)  {{\tiny X}};
\node (a) at (2,-0.75) {\small $n=4$};
\end{tikzpicture}
&
\begin{tikzpicture}[scale=0.3]
\draw[step=1cm,gray,very thin] grid (5,5);
\foreach \x/\y/\l in {1.5/2.5/a, 2.5/2.5/b, 3.5/2.5/c}
\node (\l) at (\x,\y)  {\tiny X};
\node (a) at (2.5,-0.75) {\small $n=5$};
\end{tikzpicture}
&
\begin{tikzpicture}[scale=0.3]
\draw[step=1cm,gray,very thin] grid (6,6);
\foreach \x/\y/\l in {1.5/1.5/a, 1.5/4.5/b, 4.5/1.5/c, 4.5/4.5/d}
\node (\l) at (\x,\y)  {\tiny X};
\node (a) at (3,-0.75) {\small $n=6$};
\end{tikzpicture}
&
\begin{tikzpicture}[scale=0.3]
\draw[step=1cm,gray,very thin] grid (7,7);
\foreach \x/\y/\l in {1.5/1.5/a, 1.5/5.5/b, 5.5/1.5/c, 5.5/5.5/d}
\node (\l) at (\x,\y)  {\tiny X};
\node (a) at (3.5,-0.75) {\small $n=7$};
\end{tikzpicture}
\end{array} \]

\[ \begin{array}{cccc}
\begin{tikzpicture}[scale=0.3]
\draw[step=1cm,gray,very thin] grid (8,8);
\foreach \x/\y/\l in {1.5/6.5/a, 5.5/6.5/b, 5.5/5.5/c, 2.5/2.5/d, 1.5/1.5/e, 6.5/1.5/f}
\node (\l) at (\x,\y)  {\tiny X};
\node (a) at (4,-0.75) {\small $n=8$};
\end{tikzpicture}
&
\begin{tikzpicture}[scale=0.3]
\draw[step=1cm,gray,very thin] grid (9,9);
\foreach \x/\y/\l in {3.5/7.5/a, 7.5/7.5/b, 2.5/6.5/c, 3.5/5.5/d, 6.5/2.5/e, 1.5/1.5/f, 7.5/1.5/g}
\node (\l) at (\x,\y)  {\tiny X};
\node (a) at (4.5,-0.75) {\small $n=9$};
\end{tikzpicture}
&
\begin{tikzpicture}[scale=0.3]
\draw[step=1cm,gray,very thin] grid (10,10);
\foreach \x/\y/\l in {4.5/9.5/a, 0.5/8.5/b, 8.5/8.5/c, 9.5/5.5/d, 2.5/4.5/e, 3.5/2.5/f, 8.5/1.5/g, 1.5/0.5/h}
\node (\l) at (\x,\y)  {\tiny X};
\node (a) at (5,-0.75) {\small $n=10$};
\end{tikzpicture}
\end{array} \]
\caption{Minimum exponential dominating sets of $\K_n,$ $2\le n \le10.$ }
\label{minDomsetsKing}
\end{figure}

Let $D$ be an exponential dominating set for $\K_n.$ Notice that for $d \in D,$ it follows that $|S_k(v)| = 8k$ for $k\ge 1.$ Then,
\[ w^*(d, V(\K_n) ) < 2 +  \sum_{k = 1}^\infty 8k \left( \frac{1}{2}\right)^{k-1} = 2 + \left( \frac{8}{\left(1 - \frac{1}{2}\right)^2}\right) =  34. \] This shows that $\m(\K_n) < 34.$ This fact, along with the optimal values of $k$  determined by Mixed Integer Linear Program \ref{MILP} can be applied with Lemma \ref{kbound} to determine a lower bound for $\exd(\K_n).$ Table \ref{findKtablea} for a summary of these results. Observe that for $n\ge 11,$ there is no feasible solution with Mixed Integer Linear Program \ref{MILP}. This is caused by the constraint $Ax \le b,$ since it puts a bound on the reduction of how much weight the center vertex can send out to the remaining interior vertices. Thus the best use of Mixed Integer Linear Program \ref{MILP} will occur at $n = 7.$  
\begin{table}[!htp] 
	 \caption{\label{findKtablea} Lower Bounds for $\exd(\K_n)$ for small values of $n$ }
	\begin{center}
  		\begin{tabular}{ c|c|c|c|c|c }
   			        $n$ & 3 & 5 & 7 & 9 & 11    \\ \hline &&&&& \\
			        $k$ & 1 & 5.7806 & 10.6905 & 10.4103 & $\emptyset$  \\ \hline &&&&& \\
    $\exd(\K_n) \ge$  & $\frac{n^2}{33}$ & $\frac{n^2}{28.2194}$ & $\frac{n^2}{23.3095}$ & $\frac{n^2}{23.5897}$ & $\emptyset$  \\ 
  		\end{tabular}
  	\end{center}
\end{table}

\begin{theorem}\label{bestlower}
For all $n \ge 7,$ $ \left\lceil \frac{n^2}{23.3095033018} \right\rceil  \le \exd(\K_n).$
\end{theorem}
\begin{proof}
Let $D$ be a minimum exponential dominating set for $\K_n.$ For each $d \in D$, let $H$ be the $7 \times 7$ grid centered at $d.$ The corresponding solution to Mixed Integer Linear Program \ref{MILP} gives $x_{\min} = 35.6904966982.$ Therefore let $k = 35.6904966982 - 25 = 10.6904966982$ and recall that $\m(\K_n) < 34.$ Therefore result follows from Lemma \ref{kbound}. 
\end{proof}
\begin{figure}[ht!p]
\centering
\begin{tikzpicture}[scale=0.2]
\draw[step=1cm,gray,very thin] grid (23,23);
\foreach \x/\y/\l in {0.5/22.5/a, 1.5/18.5/b, 2.5/14.5/c, 3.5/10.5/d, 4.5/6.5/e, 5.5/2.5/f}
\node (\l) at (\x,\y)  {\tiny X};
\foreach \x/\y/\l in {6.5/21.5/g, 7.5/17.5/h, 8.5/13.5/i, 9.5/9.5/j, 10.5/5.5/k, 11.5/1.5/l}
\node (\l) at (\x,\y)  {\tiny X};
\foreach \x/\y/\l in {12.5/20.5/g, 13.5/16.5/h, 14.5/12.5/i, 15.5/8.5/j, 16.5/4.5/k, 17.5/0.5/l}
\node (\l) at (\x,\y)  {\tiny X};
\foreach \x/\y/\l in {18.5/19.5/g, 19.5/15.5/h, 20.5/11.5/i, 21.5/7.5/j, 22.5/3.5/k}
\node (\l) at (\x,\y)  {\tiny X};
\end{tikzpicture}
\caption{ $T_\K,$ the $23\times 23$ exponential dominating set tile for $\K_{\infty}$ }
\label{constructionK23}
\end{figure}
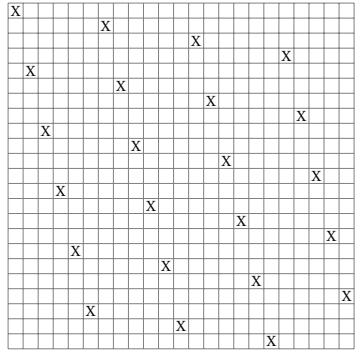

Figure \ref{constructionK23} shows a construction of a $23\times 23$ tile $T_\K,$ where `X' denotes the location of an exponential dominating vertex. In particular, when $\K_{\infty}$ is tiled with $T_\K,$ the exponential dominating set $D_\K$ is formed. The following theorem uses $T_\K$ to determines an upper bound for the asymptotic density of $\exd(\K_n).$ 

\begin{theorem}\label{UBking}{\rm
$\lim_{n\to\infty} \dfrac{\exd(\K_n)}{n^2} \le \dfrac{1}{23}$}
\end{theorem}
\begin{proof}
Let $n = 23q + r,$ for some $q,r \in \Z$ and $0\le r < 23.$ Let $H$ denote the $23q\times 23q$ subgrid of $\K_n.$ Notice that we may tile $H$ with the tiling scheme $T_{\K},$ as shown in Figure \ref{constructionK23}. Let $D_\K$ be the exponential dominating set that contains the $23q^2$ vertices used to tile $H,$ as well as $V(\K_n \setminus H).$ Therefore $\exd(\K_n) \le 23q^2 + 46qr + r^2,$ and we obtain the following asymptotic density:  

 \[ \lim_{n\to\infty} \dfrac{\exd(\K_n)}{n^2} \le \lim_{q \to\infty} \dfrac{23q^2 + 46qr + r^2}{(23q+ r)^2} \le \dfrac{1}{23} +  \lim_{q \to\infty} \dfrac{46qr + r^2}{(23q + r)^2} \le \dfrac{1}{23}, \]
 
as the limit equals zero.
\end{proof}

\begin{theorem}\label{bestupper}
For all $n \ge 23,$ $ \exd(\K_n) \le \left\lceil \frac{n^2}{23} \right\rceil + o(1).$
\end{theorem}
\begin{proof}
This result follows directly from Theorem \ref{UBking}.
\end{proof}
Similarly to Conjecture \ref{conj13}, we make the following conjecture.

\begin{conjecture}\label{conjKing}
For all $n,$ $\left\lceil \frac{n^2}{23} \right\rceil  \le \exd(\K_n).$
\end{conjecture}
\end{subsection}


\begin{subsection}{The Slant Grid $\S_n$}
Integer Program \ref{IP} can be utilized in terms of $\S_n$ to determine the exact value of $\exd(\S_n)$ for small values of $n.$ These values, as well as the locations of the exponential dominating vertices, are illustrated in Figure \ref{minDomsetsSlant}. Notice that  `X' denotes a member of $\exd(\S_n).$ See Code \ref{SlantExact} for the corresponding SAGE code.

\begin{figure}[!htp]
\[ \begin{array}{ccccccc}

&
\begin{tikzpicture}[scale=0.3]
\draw[step=1cm,gray,very thin] grid (3,3);
\foreach \x/\y/\l in {0.5/1.5/a, 2.5/1.5/b}
\node (\l) at (\x,\y)  {\tiny X};
\node (a) at (1.5,-0.75) {\small $n=3$};
\end{tikzpicture}
&
\begin{tikzpicture}[scale=0.3]
\draw[step=1cm,gray,very thin] grid (4,4);
\foreach \x/\y/\l in {0.5/2.5/a, 2.5/2.5/b, 2.5/0.5/c}
\node (\l) at (\x,\y)  {\tiny X};
\node (a) at (2,-0.75) {\small $n=4$};
\end{tikzpicture}
&
\begin{tikzpicture}[scale=0.3]
\draw[step=1cm,gray,very thin] grid (5,5);
\foreach \x/\y/\l in {0.5/3.5/a, 3.5/3.5/b, 1.5/1.5/c, 4.5/1.5/d}
\node (\l) at (\x,\y)  {\tiny X};
\node (a) at (2.5,-0.75) {\small $n=5$};
\end{tikzpicture}
&
\begin{tikzpicture}[scale=0.3]
\draw[step=1cm,gray,very thin] grid (6,6);
\foreach \x/\y/\l in {0.5/4.5/a, 4.5/4.5/b, 2.5/2.5/c, 1.5/0.5/d, 5.5/0.5/e}
\node (\l) at (\x,\y)  {\tiny X};
\node (a) at (3,-0.75) {\small $n=6$};
\end{tikzpicture}
&
\begin{tikzpicture}[scale=0.3]
\draw[step=1cm,gray,very thin] grid (7,7);
\foreach \x/\y/\l in {5.5/6.5/a, 0.5/5.5/b, 2.5/3.5/c, 5.5/2.5/d, 1.5/0.5/e, 4.5/0.5/f}
\node (\l) at (\x,\y)  {\tiny X};
\node (a) at (3.5,-0.75) {\small $n=7$};
\end{tikzpicture}
\end{array} \]

\[ \begin{array}{cccc}
\begin{tikzpicture}[scale=0.3]
\draw[step=1cm,gray,very thin] grid (8,8);
\foreach \x/\y/\l in {0.5/6.5/a, 3.5/6.5/b, 6.5/6.5/c, 6.5/3.5/d, 0.5/1.5/e, 3.5/0.5/f, 6.5/0.5/g}
\node (\l) at (\x,\y)  {\tiny X};
\node (a) at (4,-0.75) {\small $n=8$};
\end{tikzpicture}
&
\begin{tikzpicture}[scale=0.3]
\draw[step=1cm,gray,very thin] grid (9,9);
\foreach \x/\y/\l in {0.5/7.5/a, 3.5/7.5/b, 6.5/7.5/c, 8.5/6.5/d, 1.5/3.5/e, 5.5/1.5/f, 8.5/1.5/g, 1.5/0.5/h}
\node (\l) at (\x,\y)  {\tiny X};
\node (a) at (4.5,-0.75) {\small $n=9$};
\end{tikzpicture}
&
\begin{tikzpicture}[scale=0.3]
\draw[step=1cm,gray,very thin] grid (10,10);
\foreach \x/\y/\l in {8.5/9.5/a, 0.5/8.5/b, 3.5/8.5/c, 9.5/6.5/d, 4.5/5.5/e, 0.5/4.5/f, 5.5/3.5/g, 6.5/1.5/h, 9.5/1.5/i ,1.5/0.5/z}
\node (\l) at (\x,\y)  {\tiny X};
\node (a) at (5,-0.75) {\small $n=10$};
\end{tikzpicture}
\end{array} \]
\caption{Minimum exponential dominating sets of $\S_n,$ $3\le n \le10.$ }
\label{minDomsetsSlant}
\end{figure}

Let $D$ be an exponential dominating set for $\S_n.$ Notice that for $d \in D,$ we have that $|S_k(d)| \le 6k$ for $k \ge1.$ Then we can bound the total weight that $d$ sends to $V(\S_n)$ with 
 \[w^*(d, V(H_n)) < 2 + \sum_{k = 1}^ \infty 6k \left(\frac{1}{2} \right)^ {k-1} = 2 + \left( \frac{6}{\left(1 - \frac{1}{2}\right)^2}\right) = 26.\] Therefore it follows that $\m(\S_n) < 26.$

\begin{table}[h!] 
	 \caption{\label{findKtable}Lower Bounds for $\exd{(\S_n)}$ for small values of $n$ }
	\begin{center}
  		\begin{tabular}{ c|c|c|c|c }
   			        $n$ & 3 & 5 & 7 & 9     \\ \hline   &&&& \\
    $k$  & 1.2353 & 3.9774 & 6.2655 &  $\emptyset$     \\ \hline   &&&& \\
     $\exd(\S_n) \ge$  & $\frac{n^2}{24.7647}$ & $\frac{n^2}{22.0226}$ & $\frac{n^2}{19.7345}$ & $\emptyset$  
  		\end{tabular}
  	\end{center}
\end{table}
  
This fact, along with the optimal values of $k$  determined by Mixed Integer Linear Program \ref{MILP} can be applied with Lemma \ref{kbound} to determine a lower bound for $\exd(\S_n).$ Table \ref{findKtable} for a summary of these results. Observe that for $n\ge 9,$ there is no feasible solution with Mixed Integer Linear Program \ref{MILP}. This is caused by the constraint $Ax \le b,$ since it puts a bound on the reduction of how much weight the center vertex can send out to the remaining interior vertices. Thus the best use of Mixed Integer Linear Program \ref{MILP} will occur at $n = 7.$  

\begin{theorem}\label{bestlowerSn}
For all $n \ge 7,$ $ \left\lceil \frac{n^2}{19.7344975348} \right\rceil  \le \exd(\S_n).$
\end{theorem}

Figure \ref{constructionK23S19} shows a construction the $19\times 19$ tile $T_\S,$ such that when $\S_{\infty}$ is tiled with $T_\S,$ exponential dominating set $D_\S$ is formed. Notice that `X' denotes the location of an exponential dominating vertex. The following theorem uses  $T_\S$ to determine an upper bound for the asymptotic density of $\exd(\S_n).$ 

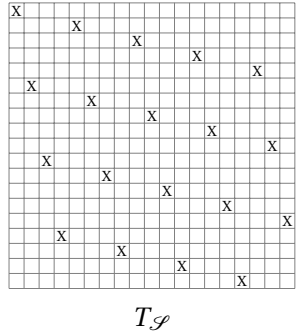
\begin{figure}[!htp]
\centering
\begin{tikzpicture}[scale=0.2]
\draw[step=1cm,gray,very thin] grid (19,19);

\foreach \x/\y/\l in {0.5/18.5/a, 1.5/13.5/b, 2.5/8.5/c, 3.5/3.5/d}
\node (\l) at (\x,\y)  {\tiny X};
\foreach \x/\y/\l in {4.5/17.5/a, 5.5/12.5/b, 6.5/7.5/c, 7.5/2.5/d}
\node (\l) at (\x,\y)  {\tiny X};
\foreach \x/\y/\l in {8.5/16.5/a, 9.5/11.5/b, 10.5/6.5/c, 11.5/1.5/d}
\node (\l) at (\x,\y)  {\tiny X};
\foreach \x/\y/\l in {12.5/15.5/a, 13.5/10.5/b, 14.5/5.5/c, 15.5/0.5/d}
\node (\l) at (\x,\y)  {\tiny X};
\foreach \x/\y/\l in {16.5/14.5/a, 17.5/9.5/b, 18.5/4.5/c}
\node (\l) at (\x,\y)  {\tiny X};

\node(aa) at (9.5, -2) {$T_\S$};

\end{tikzpicture}
\caption{ $T_\S,$ the $19\times 19$ exponential dominating set tile for $\S_{\infty}$  }
\label{constructionK23S19}
\end{figure}

\begin{theorem}\label{UBSlant}{\rm
\[ \lim_{n\to\infty} \dfrac{\exd(\S_n)}{n^2} \le \dfrac{1}{19} \]}
\end{theorem}

\begin{proof}
Let $n = 19q + r,$ for some $q,r \in \Z$ and $0\le r < 19.$ Let $H$ denote the $19q\times 19q$ subgrid of $\S_n.$ Notice that we may tile $H$ with the tiling scheme $T_{\S},$ as shown in Figure \ref{constructionK23S19}. Let $D_\S$ be the exponential dominating set that contains the $19q^2$ vertices used to tile $H,$ as well as $V(\S_n \setminus H).$ Therefore $\exd(\S_n) \le 19q^2 + 38qr + r^2,$ and we obtain the following asymptotic density:

 \[ \lim_{n\to\infty} \dfrac{\exd(\S_n)}{n^2} \le \lim_{q \to\infty} \dfrac{19q^2 + 38qr + r^2}{(19q+ r)^2} \le \dfrac{1}{19} +  \lim_{q \to\infty} \dfrac{38qr + r^2}{(19q+ r)^2} \le \dfrac{1}{19}, \]
 
as the limit equals zero.
\end{proof}

\begin{theorem}\label{bestupperSn}
For $n \ge 19,$ $ \exd(\S_n) \le \left\lceil \frac{n^2}{19} \right\rceil + o(1).$
\end{theorem}

\begin{proof}
This result follows directly from Theorem \ref{UBSlant}.
\end{proof}
Similarly to Conjecture \ref{conj13}, we make the following conjecture.

\begin{conjecture}\label{conjSlant}
For all $n,$ $\left\lceil \frac{n^2}{23} \right\rceil  \le \exd(\K_n).$
\end{conjecture}
\end{subsection}


\begin{subsection}{The $n$-dimensional hypercube}
As $Q_n$ is not a grid graph, the method used to determine a value of $k$ for Lemma \ref{kbound} in Remark \ref{LPsetup} cannot be used to find the lower bound $\exd(Q_n).$ In order to determine such a lower bound, a new method is used where distance properties of $Q_n$ are exploited.

 Let $D$ be a minimum exponential dominating set for $Q_n$ and let $d \in D.$ Observe that for $u, v \in V(Q_n),$ the length of the shortest $uv$ path in $Q_n$ can be determined by the minimum number of digits that must be changed to get from $u$ to $v.$ Then for all $v \in V(Q_n),$ we have that:
\[  w^*(v, V(Q_n)) = \sum_{i=0}^n { n \choose i} \left( \frac{1}{2} \right)^{i-1} = 2 \sum_{i=0}^n { n \choose i} \left( \frac{1}{2} \right)^i \cdot 1^{n-i} = 2\left(\frac{1}{2} + 1\right)^n = 2\left(\frac{3}{2}\right)^n. \]
Thus it follows that $\m(Q_n) = 2\left(\frac{3}{2}\right)^n.$

\begin{figure}[htp!]
\centering
\begin{tikzpicture}[scale = 1.5]
\centering
\node (a) at (-.75,0.5) {$Q_n =$};
\node (a) at (0,1) {$Q_{n-2}^{(1)}$};
\node (b) at (1,1) {$Q_{n-2}^{(2)}$};
\node (c) at (1,0) {$Q_{n-2}^{(4)}$};
\node (d) at (0,0) {$Q_{n-2}^{(3)}$};
\draw[ thick, dotted ] (a) -- (b) -- (c) -- (d) -- (a);
\end{tikzpicture}
\caption{A decomposition of $Q_n,$ where $Q_n = Q_{n-2}\square K_2 \square K_2.$}
\label{upperQn}
\end{figure}

In the following theorem, the decomposition in Figure \ref{upperQn} and value of $\m(Q_n)$ are used to show that $\left(\frac{4}{3} \right)^n \le \exd(Q_n) \le (\sqrt{2})^n$ for large $n.$

\begin{theorem} 
For all $n\ge 1,$ $ \left\lceil \dfrac{2^{n+3}}{2^{4-n} \cdot 3^n - 2n - 9 } \right\rceil  \le \exd(Q_n)\le (\sqrt{2})^{n}$
\end{theorem}

\begin{proof}
We begin with the lower bound. Let $D$ be an exponential dominating set for $Q_n$ and suppose that $d = \{0,0,\ldots, 0\} \in D.$  Let $A = \{ a$ $\in V(Q_n)$ $: \; \text{a has an odd number of } 1's\}$ and $B= V(Q_n) \setminus (A \cup d).$ Let $X = \{ x\in V(Q_n) : \; dx\in E(Q_n) \}  \subset A$ and $Y = \{ y\in V(Q_n) : \; xy\in E(Q_n) \text{ for some } x\in X \} \subset B.$ Then $w^*(d,X) = |X| = n$ and $w^*(d,Y) = \tfrac{n}{2}.$ As $(D,w^*)$ dominates $Q_n,$ $w^*(D \setminus d, Y) \ge \tfrac{n}{2}.$ This implies that $w^*(D\setminus d, X) \ge \tfrac{n}{4},$ and $w^*(D\setminus d,d) \ge \tfrac{1}{8}.$ Therefore $\exc(D, X) \ge \tfrac{n}{4}$ and $\exc(D,d) =\tfrac{9}{8},$ which holds for all $d \in D.$ This gives that
 \[ \exc(D) \ge \left( \frac{9}{8} + \frac{n}{4} \right) |D| =  \frac{2n +9}{8} |D|.\] Then using $\m(Q_n) = 2\left(\frac{3}{2}\right)^n$ and $k = \tfrac{2n + 9}{8},$ the lower bound follows from Lemma \ref{kbound}.
 
Now we show the upper bound. From Figure \ref{upperQn}, $Q_n = Q_{n-2}\square K_2 \square K_2.$ Without loss of generality, let $D$ and $D'$ be two minimum exponential dominating sets for $Q^{(1)}_{n-2}$ and $ Q^{(4)}_{n-2},$ respectively, with labeling as in Figure \ref{upperQn}. Therefore it follows by definition that $w^*(D,V(Q^{(1)}_{n-2}) )\ge 1$ and $w^*(D',V(Q^{(4)}_{n-2}) )\ge 1.$ As neighboring vertices also receive weight, $w^*(D,V(Q^{(2)}_{n-2})),$ $w^*(D,V(Q^{(3)}_{n-2}))\ge \frac{1}{2}$ and $w^*(D',V(Q^{(2)}_{n-2})),$ $w^*(D',V(Q^{(3)}_{n-2}))\ge \frac{1}{2}.$ This implies that $D\cup D'$ forms an exponential dominating set for $Q_n.$ Let $a_n = \exd(Q_n)$ and $a_{n-2} = |D|=|D'|,$ so $a_n = 2a_{n-2}.$ We now show that $a_n \le 2^{\frac{n}{2}}$ by induction. Observe that when $n = 2,$ we have that $a_2 = 2 \le 2^1.$ Now suppose that  $a_n \le 2^{\frac{n}{2}}$ holds for all $n<k.$ Now consider the case when $n = k.$ Then using the inductive hypothesis,  \[ a_k = 2a_{k-2}\le 2(2^{\frac{1}{2}(k-2)}) = 2^{\frac{k}{2}}. \] Therefore by induction, $ \exd(Q_n) \le (\sqrt{2})^{n}.$
 \end{proof}

\end{subsection}
\end{section}


\section{Acknowledgements}
This research was supported in part by the National Science Foundation Award $\#$ 1719841. We would like to thanks Dr. Leslie Hogben for her input on this paper.

\section{Appendix}
This sections consists of the SAGE code referenced throughout \emph{A linear programming method for exponential domination}.

\begin{code} \label{KingExact} Integer program that computes $\gamma^*_e(\K_n)$ for small values of $n$

\begin{lstlisting}[language=Python, basicstyle=\scriptsize]
for n in range(1,11):
    G = graphs.KingGraph( [n,n] ) # Sets up King grid
    KingDist = G.distance_matrix() # King grid distance matrix
    m = KingDist.nrows()
    
    # This matrix represents the weight that 
    # vertex i sends to vertex j in the King grid
    K = matrix(QQ, m,m, lambda i, j: (1/2)^(KingDist[i][j]-1)) 
    
    # Sets up the MILP
    p = MixedIntegerLinearProgram(maximization=False, solver="GLPK");
    x = p.new_variable(integer=True, nonnegative=True)
    
    # creates the objective function 
    s = 0
    for i in range(m):
        p.add_constraint(x[i] <= 1)    
        s = s + x[i]   
    
    p.set_objective(s) 
    p.add_constraint(K*x >= 1)   
    print n, p.solve()
    \end{lstlisting}
    \end{code}
    
    \begin{code}\label{KingLP} Mixed integer linear program to find upper bound of $\exd(\K_n).$
\begin{lstlisting}[language=Python, basicstyle=\scriptsize]
# This function computes the distance between two vertices in the King grid
def distance(a,b):
    ans = max( abs(a[0] - b[0]), abs(a[1] - b[1]) )
    return ans

# The Linear Program for the King grid
for n in range (3, 13, 2):
    p = MixedIntegerLinearProgram(maximization=False, solver="GLPK");
    x = p.new_variable(real=True, nonnegative=True)   
    King = graphs.KingGraph([n,n]) # Generated the nxn king grid
    K = King.vertices()
    
    A = zero_matrix(RR,n^2,n^2) # Creates the A matrix for LP
    for i in range(n^2):
        for j in range(n^2):
            A[i, j] = (1/2)^ distance(K[i], K[j] ) 
    
    KingDist =  zero_matrix(RR,n^2,n^2) # Creates the distance matrix for King grid
    for i in range(n^2):
        for j in range(n^2):
            KingDist[i,j] = distance(K[i],K[j])  
    
    b = [34]*n^2 # creates the b vector for King grid
    for i in range(0,n^2):
        if divmod(i,n)[1] != 0 and divmod(i,n)[1]!= n-1:   
            b[i] = 1+ (1/2)^(KingDist[(n^2-1)/2, i] - 1)
    
    # Finds the inner vertices of King grid
    list=[]
    for i in range(n,n^2-n):
        if divmod(i,n)[1] != 0 and divmod(i,n)[1]!= n-1:   
            list = list + [i]
    #make the matrix using only the rows that are inner vertices        
    c =A.matrix_from_rows(list) 
    sum=0
    
    one_vec = [1]*n^2    
    # Adds in the constraint that 1<= Ax 
    p.add_constraint(A*x >= one_vec) 
    # Adds the constraint that the weight of the middle vertex is 2
    p.add_constraint(x[(n^2-1)/2] == 2)           
   # Adds in the constraint that Ax <= b
    p.add_constraint(A*x <= b)  
    
    for i in range(len(list)):
        p.set_integer(x[list[i]])
       # Adds in constraint that inner vertices have weight 0 or 2     
        p.add_constraint(x[list[i]] <= 2) 
    
    # Sets the objective function
    for i in range(c.nrows()):
        for j in range(c.ncols()):
            sum = sum + c[i][j]*x[j]      
    
    # Computes the minimum exponential dominating number
    p.set_objective(sum)
    ans = 34 + (n-2)^2 - p.solve();
    print n, ans
    \end{lstlisting}
    \end{code}

\begin{code} \label{SlantExact} Integer program that computes $\gamma^*_e(\S_n)$ for small values of $n$
\begin{lstlisting}[language=Python, basicstyle=\scriptsize]
for n in range(2,11):
    Slant = SlantGrid(n)
    Sverts = Slant.vertices()
    
# This matrix represents the weight that 
# vertex i sends to vertex j in the Slant grid    
    A = zero_matrix(QQ,n^2,n^2)      
    for i in range(n^2):
        for j in range(n^2):
            A[i, j] = (1/2)^(Slant.distance(Sverts[i], Sverts[j] ) -1)
# Constructs integer program                   
    p = MixedIntegerLinearProgram(maximization=False, solver="GLPK");
    x = p.new_variable(integer=True, nonnegative=True)
    s = 0
    for i in range(n^2):
        p.add_constraint(x[i] <= 1)    
        s = s + x[i]   
    
    p.set_objective(s) 
    p.add_constraint(A*x >= 1)   
    print n, p.solve()
    \end{lstlisting}
    \end{code}

\begin{code}\label{SlantLP} Mixed integer linear program to find a lower bound of $\exd(\S_n).$
\begin{lstlisting}[language=Python, basicstyle=\scriptsize]
def SlantGrid(n):
    S = graphs.Grid2dGraph(n,n)
    for i in range(1, n ):
        for j in range(n-1):
            S.add_edge( (i,j) , (i-1, j+1)  )
    return S
    
# The Mixed Integer Linear Program for the Slant grid 

for n in range(3, 20, 2):
    p = MixedIntegerLinearProgram(maximization=False, solver="GLPK");
    x = p.new_variable(real=True, nonnegative=True)   
    
    Slant = SlantGrid(n) # Generated the nxn king grid
    Sverts = Slant.vertices()
    
    A = zero_matrix(RR,n^2,n^2) # Creates the A matrix for LP
    for i in range(n^2):
        for j in range(n^2):
            A[i, j] = (1/2)^Slant.distance(Sverts[i], Sverts[j] ) 
    one_vec = [1]*n^2
     # Creates the distance matrix for King grid
    SlantDist =  zero_matrix(RR,n^2,n^2) 
    for i in range(n^2):
        for j in range(n^2):
            SlantDist[i,j] = Slant.distance(Sverts[i],Sverts[j])  
    
    b = [26]*n^2
    for i in range(0,n^2):
        if divmod(i,n)[1] != 0 and divmod(i,n)[1]!= n-1:   
            b[i] = 1+ (1/2)^(SlantDist[(n^2-1)/2, i] - 1)
    
    # Finds the inner vertices of Slant grid
    list=[]
    for i in range(n,n^2-n):
        if divmod(i,n)[1] != 0 and divmod(i,n)[1]!= n-1:   
            list = list + [i]
    #make the matrix using only the rows that are inner vertices        
    c =A.matrix_from_rows(list) 
    sum=0
    
    # Adds in the constraint that 1<= Ax 
    p.add_constraint(A*x >= one_vec) 
   # Adds the constraint that the weight of the center vertex is 2 
    p.add_constraint(x[(n^2-1)/2] == 2)          
   # Adds in the constraint that Ax <= b 
    p.add_constraint(A*x <= b) 
    
    for i in range(len(list)):
        p.set_integer(x[list[i]])
        # Adds in constraint that inner vertices have weight 0,1, or 2
        p.add_constraint(x[list[i]] <= 2) 
    
    # Sets the objective function
    for i in range(c.nrows()):
        for j in range(c.ncols()):
            sum = sum + c[i][j]*x[j]      
    
    # Computes the minimum exponential dominating number
    p.set_objective(sum)
    ans = 26 + (n-2)^2 - p.solve();
    print n, ans, p.solve() - (n-2)^2
    \end{lstlisting}
    \end{code}
    
    \begin{code} \label{CubeExact} Integer program that computes $\gamma^*_e(Q_n)$ for small values of $n$
\begin{lstlisting}[language=Python, basicstyle=\scriptsize]
for m in range (1, 8):
    g = graphs.CubeGraph(m)
    M = g.distance_matrix()

    # sets up the integer program		
    p = MixedIntegerLinearProgram(maximization=False, solver="GLPK");
    n = M.nrows()

    # This matrix represents the weight that 
    # vertex i sends to vertex j in the hypercube
    K = matrix(QQ, n,n, lambda i, j: (1/2)^(M[i][j]-1))
    x = p.new_variable(integer=True, nonnegative=True)
    s = 0
    for i in range(n):
        p.add_constraint(x[i] <= 1)    
        s = s + x[i]   
        
    # sets the objective function and constraints	   
    p.set_objective(s) 
    p.add_constraint(K*x >= 1)   
    print m, p.solve()
    \end{lstlisting}
    \end{code}


\bibliographystyle{plain}

\begin{thebibliography}{99}

\bibitem{anderson}
Mark Anderson, Robert~C. Brigham, Julie~R. Carrington, Richard~P. Vitray, and
  Jay Yellen.
\newblock On exponential domination of {$C_m\times C_n$}.
\newblock {\em AKCE Int. J. Graphs Comb.}, 6(3):341--351, 2009.

\bibitem{Ayta}
A.~Ayta\c{c} and B.~Atay.
\newblock On exponential domination of some graphs.
\newblock {\em Nonlinear Dyn. Syst. Theory}, 16(1):12--19, 2016.

\bibitem{Bessy1}
St{\'e}phane Bessy, Pascal Ochem, and Dieter Rautenbach.
\newblock Exponential domination in subcubic graphs.
\newblock {\em Electr. J. Comb.}, 23:P4.42, 2016.

\bibitem{Bessy}
St\'ephane Bessy, Pascal Ochem, and Dieter Rautenbach.
\newblock Bounds on the exponential domination number.
\newblock {\em Discrete Math.}, 340(3):494--503, 2017.

\bibitem{EGR}
Chassidy Bozeman, Joshua Carlson, Michael Dairyko, Derek Young, and Michael
  Young.
\newblock Lower bounds for the exponential domination number of ${C_m \times
  C_n}$.
\newblock https://orion.math.iastate.edu/mdairyko/LowerBounds.pdf.

\bibitem{dankel}
Peter Dankelmann, David Day, David Erwin, Simon Mukwembi, and Henda Swart.
\newblock Domination with exponential decay.
\newblock {\em Discrete Math.}, 309(19):5877--5883, 2009.

\bibitem{Haynes}
Teresa~W. Haynes, Stephen~T. Hedetniemi, and Peter~J. Slater.
\newblock {\em Fundamentals of domination in graphs}, volume 208 of {\em
  Monographs and Textbooks in Pure and Applied Mathematics}.
\newblock Marcel Dekker, Inc., New York, 1998.

\bibitem{Henning1}
Michael~A. Henning, Simon J\"ager, and Dieter Rautenbach.
\newblock {Hereditary Equality of Domination and Exponential Domination}.
\newblock {\em Discussiones Mathematicae Graph Theory}, 2016.

\bibitem{Henning}
Michael~A. Henning, Simon J\"ager, and Dieter Rautenbach.
\newblock Relating domination, exponential domination, and porous exponential
  domination.
\newblock {\em Discrete Optim.}, 23:81--92, 2017.

\end{thebibliography}

 
\end{document}